\documentclass[12pt]{article}

\usepackage{geometry}                
\usepackage{graphicx}
\usepackage{amssymb}
\usepackage{epstopdf}
\usepackage{latexsym, amssymb, amsmath, amsthm, graphics}
\usepackage[mathscr]{eucal}
\usepackage{url}
\usepackage{eucal}

\title{ \bf Stable Tameness of Two-Dimensional Polynomial Automorphisms Over a Regular Ring} \author{Joost Berson, Arno van den Essen, and David Wright}

\begin{document}
    \maketitle

 \input amssym.def
    \input amssym
    \theoremstyle{definition} 
    \newtheorem{rema}{Remark}[section]
    \newtheorem{questions}[rema]{Questions}
    \newtheorem{assertion}[rema]{Assertion}
    \newtheorem{num}[rema]{}
           \newtheorem{exam}[rema]{Example}
        \newtheorem{ques}[rema]{Question}
    \newtheorem{claim}[rema]{Claim}
    \theoremstyle{plain} 
    \newtheorem{propo}[rema]{Proposition}
    \newtheorem{theo}[rema]{Theorem} 
    \newtheorem{conj}[rema]{Conjecture}
    \newtheorem{quest}[rema]{Question} 
    \theoremstyle{definition}
    \newtheorem{defi}[rema]{Definition}
    \theoremstyle{plain}
    \newtheorem{lemma}[rema]{Lemma} 
    \newtheorem{corol}[rema]{Corollary}
    \newtheorem{rmk}[rema]{Remark}
    \newcommand{\del}{\triangledown}
    \newcommand{\nno}{\nonumber} 
    \newcommand{\lbar}{\big\vert}
    \newcommand{\mbar}{\mbox{\large $\vert$}} 
    \newcommand{\p}{\partial}
    \newcommand{\dps}{\displaystyle} 
    \newcommand{\bra}{\langle}
    \newcommand{\ket}{\rangle} 
    \newcommand{\kr}{\mbox{\rm Ker}\ }
    \newcommand{\res}{\mbox{\rm Res}} 
    \renewcommand{\hom}{\mbox{\rm Hom}}
    \newcommand{\pf}{{\it Proof:}\hspace{2ex}}
    \newcommand{\epf}{\hspace{2em}$\Box$}
    \newcommand{\epfv}{\hspace{1em}$\Box$\vspace{1em}}
    \newcommand{\nord}{\mbox{\scriptsize ${\circ\atop\circ}$}}
    \newcommand{\wt}{\mbox{\rm wt}\ } 
    \newcommand{\clr}{\mbox{\rm clr}\ }
    \newcommand{\ideg}{\mbox{\rm Ideg}\ } 
    \newcommand{\gC}{{\mathfrak g}_{\mathbb C}}
    \newcommand{\hatC}{\widehat {\mathbb C}} 
    \newcommand{\bZ}{{\mathbb Z}} 
    \newcommand{\bQ}{{\mathbb Q}}
    \newcommand{\bR}{{\mathbb R}} 
    \newcommand{\bN}{{\mathbb N}}
    \newcommand{\bT}{{\mathbb T}} 
    \newcommand{\fg}{{\mathfrak g}}
    \newcommand{\fgC}{{\mathfrak g}_{\bC}} 
    \newcommand{\cD}{\mathcal D} 
    \newcommand{\cP}{\mathcal P}
    \newcommand{\cC}{\mathcal C} 
    \newcommand{\cS}{\mathcal S}
    \newcommand{\EGC}{{\cal E}(\GC)}
    \newcommand{\cLGC}{\widetilde{L}_{an}\GC}
    \newcommand{\LGC}{{L}_{an}\GC} 
    \newcommand{\BQ}{\begin{eqnarray}}
    \newcommand{\EQ}{\end{eqnarray}} 
    \newcommand{\BQn}{\begin{eqnarray*}}
    \newcommand{\EQn}{\end{eqnarray*}} 
    \newcommand{\wtilde}{\widetilde}
    \newcommand{\Hol}{\mbox{Hol}} 
    \newcommand{\Hom}{\mbox{Hom}}
    \newcommand{\poly}{polynomial } 
    \newcommand{\polys}{polynomials }
    \newcommand{\pz}{\frac{\p}{\p z}} 
    \newcommand{\pzi}{\frac{\p}{\p z_i}}
    \newcommand{\edge}{\text{\raisebox{2.75pt}{\makebox[20pt][s]{\hrulefill}}}}
    \newcommand{\halfedge}{\text{\raisebox{2.75pt}{\makebox[10pt][s]{\hrulefill}}}}
    \newcommand{\n}{\notag}
    \newcommand{\C}{\mathbb C}
    \newcommand{\A}{\mathcal A}
    \newcommand{\Q}{\mathbb Q}
    \newcommand{\X}{X_1,\ldots,X_n}
    \newcommand{\Xmi}{X,\hat i}
    \newcommand{\Z}{Z_1,\ldots,Z_n}
     \newcommand{\Zmi}{Z,\hat i}
    \newcommand{\pa}{\partial}
    \newcommand{\D}{D_1,\ldots,D_n}
    \newcommand{\Del}{\text{\raisebox{2.1pt}{$\bigtriangledown$}}}
    \newcommand{\La}{\triangle}
    \newcommand{\Hess}{\text{\rm Hess}}
\newcommand{\F}{F_{1},\ldots,F_{n}}
\newcommand{\G}{G_{1},\ldots,G_{n}}
\newcommand{\GA}{\text{GA}}
\newcommand{\SA}{\text{SA}}
\newcommand{\MA}{\text{MA}}
\newcommand{\GL}{\text{GL}}
\newcommand{\SL}{\text{SL}}
\newcommand{\GE}{\text{GE}}
\newcommand{\Tr}{\text{Tr}}
\newcommand{\Af}{\text{Af}}
\newcommand{\Bf}{\text{Bf}}
\newcommand{\W}{\text{W}}
\newcommand{\EA}{\text{EA}}
\newcommand{\BA}{\text{BA}}
\newcommand{\E}{\text{E}}
\newcommand{\B}{\text{B}}
\newcommand{\T}{\text{TA}}
\newcommand{\Di}{\text{D}}
\newcommand{\WT}{\text{WTA}}
\newcommand{\degr}{\text{deg}\,}
\newcommand{\supp}{\text{supp}\,}
\newcommand{\ideaala}{\mathfrak{a}}
\newcommand{\ideaalb}{\mathfrak{b}}
\newcommand{\ideaalc}{\mathfrak{c}}
\newcommand{\m}{\mathfrak{m}}
\newcommand{\id}{\text{id}}
\abstract{In this paper it is established that all two-dimensional polynomial automorphisms over a regular ring $R$ are stably tame.  This results from the main theorem of this paper, which asserts that an automorphism in any dimension $n$ is stably tame if said condition holds point-wise over Spec$\,R$.  A key element in the proof is a theorem which yields the following corollary: Over an Artinian ring $A$ all two-dimensional polynomial automorphisms having Jacobian determinant one are stably tame, and are tame if $A$ is a $\Q$-algebra.  Another crucial ingredient, of interest in itself, is that stable tameness is a local property:  If an automorphism is locally tame, then it is stably tame.}


\setcounter{equation}{0}
\setcounter{rema}{0}

\section{Introduction} The famous theorem of Jung and Van der Kulk (\cite{Jung},\cite{vdK}) asserts that all two-dimensional polynomial automorphisms over a field are tame. (See \S2 for the definition of tameness and other terminology.)  Jung proved this for fields of characteristic zero and Van der Kulk generalized it to arbitrary characteristic.  It is well-known that this fails to be true over a domain $R$ which is not a field.  A standard example of a non-tame automorphism is $$\left(X+a(aY+X^2),Y-2X(aY+X^2)-a(aY+X^2)^2\right)$$ where $a$ is any non-zero non-unit in $R$. For $R=k[T]$ and $a=T$, $k$ a field, this is the famous example of Nagata \cite{Nag} which he conjectured to be non-tame as a 3-dimensional automorphism over $k$.  Shestakov and Umirbaev \cite{S-U} finally proved Nagata's conjecture.\footnote{The proof depends on a crucial inequality established in \cite{S-U2}, a result generalized and clarified by Kuroda in \cite{Kur}.} Meanwhile it had been shown by Smith \cite{Smith} and Wright (unpublished) that Nagata's example is \underbar{stably} tame, in fact tame with the addition of one more variable.\footnote{Smith's method uses the fact that this automorphism is the exponential of a locally nilpotent derivation.  However, one can modify Nagata's example slightly so that it does not appear to be such an exponential, but still becomes tame with one new dimension.}  The matter of stable tameness is one of intrigue because no example has been produced (to the authors' knowledge) of a polynomial automorphism over a domain which cannot be shown to be stably tame.

The remarkable result of Umirbaev and Shestakov mentioned above actually asserts that an automorphism in three variables $T,X,Y$ over a field $k$ which fixes $T$ is tame (if and) only if it is tame as an automorphism over
 $k[T]$.  As there are known to be many non-tame two-dimensional automorphisms over $k[T]$, this establishes the existence of many non-tame three-dimensional automorphisms over $k$.  However, it will follow from the main result of this paper (Corollary \ref{cor}) that {\it all three-dimensional automorphisms of this type are stably tame over $k$.}

The main result of this paper is Theorem \ref{main0} (Main Theorem), which asserts that all two-dimensional polynomial automorphisms over a regular ring are stably tame.  It is proved by a somewhat delicate argument for which Theorem \ref{modnil} plays an essential role.  The latter result  yields the consequence that all two-dimensional automorphisms over an Artinian ring $A$ are stably tame, Theorem \ref{dimzerotame}.  Moreover, they are actually tame in the case $A$ is a $\Q$-algebra.  The latter statement can be viewed as a generalization of Jung's Theorem, and it yields a stronger version of the Main Theorem for the case of a Dedekind $\Q$-algebra (Theorem \ref{submain}).  Another keystone in the proof of the Main Theorem is Theorem \ref{loc}, which reveals stable tameness to be a local property.

Also used in the proof of the Main Theorem are the Jung-Van der Kulk Theorem, a number of technical results, and a theorem of Suslin, all of which appear in \S\ref{prelims}.  Stable tameness has the flavor of $K$-theory, and some of the tools are suggestive of those used to prove results about the behavior of the functor $K_1$ under polynomial extensions (compare Lemma \ref{linconjelem}, for example, with Suslin's Lemma 3.3 in \cite{Sus}).  

We here note that the appeal to Suslin's theorem (Theorem \ref{suslin}) is precisely where the hypothesis $A$ is regular is required.  This is evoked to conclude the proof of Theorem \ref{locmod}, on which the Main Theorem depends.  The Main Theorem certainly fails for non-reduced rings, over which there exist automorphisms whose Jacobian determinant lies outside the base ring -- a property which precludes stable tameness.  Beyond that we have not investigated the extent to which the regularity hypothesis can be relaxed (if at all).

This paper is organized as follows:  The basic definitions and facts surrounding automorphisms and automorphism groups are laid out in \S\ref{notation}.  Subsequently, \S\ref{prelims} presents most of the technical tools needed for the main results, which are then stated and proved in \S\ref{mainresults}.  The Main Theorem (Theorem \ref{main0}) is proved by a series of reductions to other assertions which are stated and proved as they are encountered rather than in \S\ref{prelims} in order to make the argument more transparent to the reader.

\section{Notation, Terminology, and First Observations}\label{notation}

\begin{num}\label{begin}In this paper ``ring" will mean ``commutative ring with identity". For $R$ a ring we sometimes write $R^{[n]}$ for the polynomial ring $R[\X]$.  We will often need to refer to the subalgebra $R[X_1,\ldots,X_{i-1},X_{i+1},\ldots,X_n]$ for $i\in\{1,\ldots,n\}$, so we will use the shorter notation $R[\Xmi]$ to denote the latter.
\end{num}

\begin{num}  The symbol $ \GA_n(R)$ denotes the {\it general automorphism group, }by which we mean the automorphism group of $\text{Spec}\,R^{[n]}$ over $\text{Spec}\,R$.  As such, it is anti-isomorphic to the group of $R$-algebra automorphisms of $R^{[n]}$.   An element of $\GA_n(R)$ is represented by a vector $\varphi=(\F)\in(R^{[n]})^n$; we will consistently use Greek letters to denote automorphisms.  The variables being used in the vector representation of elements of $\GA_n(R)$ ($\X$ at the moment) are called {\it dimension variables} to distinguish them from variables that may be a part of the coefficient ring $R$, which may itself be a polynomial ring.  We often write $X$, or $\id_n$, or simply $\id$, for the identity element $(\X)$ of $\GA_n(R)$; we also will sometimes use such vector notation for an arbitrary polynomial map or for a system of variables, e.g., $H$ for $(H_1,\ldots,H_n)\in (R^{[n]})^n$ or $Y$ for variables $Y_1,\ldots,Y_n$.  We write $J\varphi$ for the Jacobian matrix of an automorphism $\varphi$.  
\end{num}

\begin{num}\label{defs}We have the following subgroups of $\GA_n(R)$ (and we here suppress $R$):
\begin{itemize}
\item The {\it general linear group} $\GL_n$ is contained in $\GA_n$ in an obvious way.  If $\alpha\in\GL_n(R)$ has matrix representation $\mathcal A$, then $\alpha$ has the vector representation 
$$
\left(\mathcal A\cdot X^\text{t}\right)^\text{t}
$$
for which we will engage in a slight abuse of notation by suppressing the transposes and writing simply $\mathcal AX$. We will use standard notation for the other linear groups, such as $\SL_n$, $\E_n$ (the subgroup of $\GL_n$ generated by elementary matrices), $\Di_n$ (the group of invertible diagonal matrices), and $\GE_n$ (the subgroup generated by $\E_n$ and $\Di_n$).  
\item $\SA_n$, the {\it special automorphism group, }is the subgroup of all $\varphi$ for which $|J\varphi|=1$.  (Here and throughout this paper $|\,\,\,|$ denotes determinant.)
\item $\EA_n$ is the subgroup generated by the elementary automorphisms.  An {\it elementary} automorphism is one of the form 
\begin{equation}\label{e} e_i(f)=(X_1,\ldots,X_{i-1},X_i+f,X_{i+1},\ldots,X_n)\end{equation} 
for some $i\in\{1,\ldots,n\}$, $f\in R[\Xmi]$ (see \ref{begin} for notation).  An elementary automorphism of the above form for a specific $i$ is called {\it elementary in the $i^\text{th}$ position.} One quickly verifies that $e_i$ is a group homomorphism from the additive group of $R[\Xmi]$ to $\GA_n$:
\begin{equation}e_i(f+g)=e_i(f)\circ e_i(g)
\end{equation}
This notation is suggestive of the symbol $e_{ij}(a)$ ($i\ne j$) from linear algebra, which denotes the elementary matrix having $a$ in the $ij$ position, 1 in each diagonal position, and 0 elsewhere.  By the inclusion of $\GL_n$ in $\GA_n$ we have $e_{ij}(a)=e_i(aX_j)$.  Hence $\E_n\subseteq\EA_n$.  It is not difficult to see that $\E_n=\EA_n\cap\GL_n$.  Also note that $\EA_n\subseteq \SA_n$.
\item $\T_n$, the group of {\it tame} automorphisms, is the subgroup generated by $\GL_n$ and $\EA_n$.  Over a domain these are the only obvious examples of polynomial automorphisms, though we know others exist.  A fundamental issue -- one which this paper addresses -- is the matter of determining when automorphisms are tame.
\item $\Tr_n$ is the subgroup of translations.  A {\it translation} is an automorphism of the form $X+v=(X_1+v_1,\ldots,X_n+v_n)$ with $v=(v_1,\ldots,v_n)\in R^n$.  This group is isomorphic to the additive group $R^n$ via the map $v\mapsto X+v$, for $v\in R^n$. 
\item $\Af_n$, the {\it affine} group, is the subgroup generated by $\GL_n$ and $\Tr_n$.  It is, in fact, the semidirect product $\GL_n\ltimes\Tr_n$, with $\GL_n$ acting by conjugation on $\Tr_n\cong R^n$ in the obvious way.  Namely, for $\alpha\in\GL_n$ and $v\in R^n$,
\begin{equation}\label{conjugate}\alpha\circ(X+v)\circ\alpha^{-1}=X+\left((\alpha\cdot v^\text{t})^\text{t}\right)\,,\end{equation}
where $v^\text{t}$ is $v$ written as a column vector and $\alpha\cdot v^\text{t}$ is matrix multiplication.
\item $\GA_n^0$ is the subgroup of {\it origin preserving} automorphisms, i.e., those of the form $\varphi=(\F)$ with $F_i(0,\ldots,0)=0$ for $i=1,\ldots,n$.  Clearly $\GA_n^0$ contains $\GL_n$.
\end{itemize}
\end{num}

\begin{num}[$\delta$ notation]\label{delta} It will be convenient, when $n$ is understood, to write $\delta_i$ for the $n$-dimensional vector $(0,\ldots,0,1,0,\ldots,0)$ with the 1 in the $i^\text{th}$ position.   Note then, that the elementary automorphism $e_i(f)$ of (\ref{e}) can be written using vector notation as
$e_i(f)=X+f\delta_i$.
\end{num}

\begin{num} If $G$ and $H$ are subgroups of some group, we write $\langle G,H\rangle$ for the subgroup generated by $G\cup H$.  For example $\T_n=\langle\GL_n,\EA_n\rangle$ and $\GE_n=\langle\Di_n,\E_n\rangle$.  
\end{num}

\begin{num} For any subgroup $G$ of $\GA_n$,  we write $G^0$ for $G\cap \GA_n^0$.  Thus we have $\T_n^0$, $\EA_n^0$, etc.  One easily verifies that $\EA_n^0$ is generated by elementary automorphisms of the type $e_i(f)$ where $f$ has 0 constant term, and that $\T_n^0=\langle\,\GL_n\,,\,\EA_n^0\,\rangle$.
\end{num}

\begin{defi}\label{tamequiv} We say $\varphi,\psi\in\GA_n(R)$ are {\it tamely equivalent} (respectively {\it elementarily equivalent}) if there exist $\epsilon,\epsilon'$ in $\T_n(R)$ (resp.\  $\EA_n(R)$) such that $\epsilon \varphi\epsilon'= \psi $.  To show that an automorphism is tame (resp.\   a product of elementaries) we may replace it by an automorphism to which it is tamely (resp.\ elementarily) equivalent.
\end{defi}

\begin{num}[Base change]\label{modgp} All of the groups defined in \ref{defs} can be viewed as functors.  A ring homomorphism $R\to S$ induces a group homomorphism $\GA_n(R)\to\GA_n(S)$ in a functorial way, and the same holds replacing $\GA_n$ with any of the subgroups defined above.  
\begin{enumerate}
\item We will often encounter the case where $S=R/I$ for some ideal $I\subseteq R$.  In this situation we will often write $\bar\varphi$ for the image of $\varphi\in\GA_n(R)$ in $\GA_n(R/I)$.  
\item If $t\in R$ we write $R_t$ for the localization $R[1/t]$ of $R$, and write $\varphi_t$ for the image of $\varphi$ in $\GA_n(R_t)$.  \item In the case where $R$ is a polynomial ring $A[Z_1,\ldots,Z_r]$ we will sometimes denote an element $\varphi\in\GA_n(R)$ by $\varphi(Z_1,\ldots,Z_r)$ as this allows us to write $\varphi(z_1,\ldots,z_r)$ for the base change that specializes $Z_i$ to $z_i$, where $z_1,\ldots,z_r$ lie in some $A$-algebra.   
\end{enumerate}
\end{num}

\begin{defi}\label{vanishing} In the situation of 3 above, we say that $\varphi\in\GA_n(A[Z_1,\ldots,Z_r])$ is $Z_j${\it -vanishing} if $\varphi(Z_1,\ldots,Z_{j-1},0,Z_{j+1},\ldots,Z_r)=\id_n$.
\end{defi}

\begin{num}[Lifting elementary automorphisms]\label{lift} If $R\to \bar R$ is a surjective ring homomorphism, then any elementary automorphism $\bar\rho$ over $\bar R$ lifts to an elementary automorphism $\rho$ over $R$.  It follows that the base change homomorphism $\EA_n(R)\to\EA_n(\bar R)$ is surjective.
\end{num}

\begin{num}[Stabilization]\label{stabdef}The results herein involve the concept of {\it stabilization,} which refers to the embedding of $\GA_n(R)$ into $\GA_{n+m}(R)$ (the ``stabilization homomorphism").  If $\varphi=F=(\F)\in\GA_n(R)$, we write $\varphi^{[m]}$ for its image $(\F,X_{m+1},\ldots,X_{n+m})=(F,\id_m)$ in $\GA_{n+m}(R)$; we also sometimes just write $\varphi$ for $\varphi^{[m]}$.  We say, for example, an automorphism $\varphi$ is {\it stably tame} if it becomes tame in some higher dimension.  We sometimes specify the number of dimensions by saying ``$\varphi$ becomes tame with the addition of $m$ dimensions (or variables)", meaning $\varphi^{[m]}$ is tame.  
\end{num}

\begin{num}[Direct limit]\label{dirlim} Stabilization (\ref{stabdef}) give us a chain of containments $$\GA_1\subset\GA_2\subset\GA_3\subset\cdots\,.$$  In the spirit of algebraic K-theory, we can form the direct limit, or formal ascending union, which we denote by $\GA_\infty$.  We can do the same with the other groups defined in \ref{defs}, so we have $\EA_\infty$, $\T_\infty$, etc.

\end{num}

\begin{num}[Restriction/extension of scalars]We will also encounter the ``restriction of scalars" embedding, by which we view $\GA_m(R^{[n]})$ as the subgroup of $\GA_{n+m}(R)$ which fixes (anti-isomorphically) the first $n$ variables.  By this identification we have $\EA_m(R^{[n]})\subset\EA_{n+m}(R)$, but the embedding does \underbar{not} automatically place $\T_m(R^{[n]})$ within $\T_{n+m}(R)$; we do not know whether this containment holds.  There are situations where elements of $\GL_m(R^{[n]})$ do not appear to be be tame over $R$.  This enigma presents an obstruction in the proof of Theorem \ref{locmod} which requires the use of Theorem \ref{suslin} (Suslin) to surmount.
\end{num}

\begin{num}[Products of rings]\label{prod} If a ring $R$ is a direct product of rings $R=R_1\times R_2$, then $\GA_n(R)$ is canonically isomorphic to the direct product of groups $\GA_n(R_1)\times\GA_n(R_2)$, and the same holds replacing $\GA$ by any of the subgroup functors defined in \ref{defs}.
\end{num}

\begin{num}[Scalar operator]\label{scalarop}Our results will require a scalar operator which applies only to origin preserving automorphisms.  Given $\varphi\in\GA_n^0(R)$, $t\in R$,  we define $\varphi^t\in\GA_n^0(R)$ as follows:  Write $\varphi=F_{(1)}+F_{(2)}+\cdots$ where $F_{(d)}$ is homogeneous of degree $d$.  We let $$\varphi^t=F_{(1)}+tF_{(2)}+t^2F_{(3)}+\cdots\,.$$  The following properties are easily verified:
\begin{itemize}
\item The map $\varphi\mapsto\varphi^t$ is a group endomorphism on $\GA_n^0(R)$, and this defines an action of multiplicative monoid $R$ on $\GA_n^0(R)$.
\item This action fixes elements of $\GL_n(R)$.
\item If $t\in R^*$, then $\varphi^t=\tau^{-1}\varphi\tau$, where $\tau=(tX_1,\ldots,tX_n)$.
\item We have $\varphi^0\in\GL_n(R)$, and this is just the linear homogeneous part of $\varphi$.
\end{itemize}
\end{num}

\section{Preliminaries}\label{prelims}

First we state the classical theorem which was mentioned in the introduction.

\begin{theo}[Jung-Van der Kulk \cite{Jung},\cite{vdK}] \label{JvdK}  For $k$ be a field we have $\T_2(k)=\GA_2(k)$.
\end{theo}

This rest of this section will present some technical tools needed in the proofs of the main results.  Some of these are of intrinsic interest, but others may seem unmotivated until one sees their application.  Hence the reader may prefer to read them as they are encountered in \S\ref{mainresults}.

Throughout this section $R$ will denote a (commutative) ring. 

The statement of the following lemma appears in \cite{vdE}, \S~5.2, as Exercise~7.

\begin{lemma} \label{linearcomb}
Let $R$ be a $\Q$-algebra and $X$ and $Y$ two variables. Then every
monomial $X^nY^m$ in the polynomial ring $R^{[2]}=R[X,Y]$ can be written as
a $\Q$-linear combination of polynomials of the form $(X+aY)^{n+m}$,
with $a \in \Q$.
\end{lemma}

\begin{proof}
For every $k \in \{0,\ldots\!,n+m\}$, we have the identity
$$
(X+kY)^{n+m}=\sum_{i=0}^{n+m}
\tbinom{n+m}{i}k^iX^{n+m-i}Y^i
$$
Now define vectors $v,w \in R[X,Y]^{n+m+1}$ by
\begin{align}v&=\left(X^{n+m},(X+Y)^{n+m},(X+2Y)^{n+m},\ldots\!,(X+(n+m)Y)^{n+m}\right)\notag\\
w&=\left(\tbinom{n+m}{0}X^{n+m}, \tbinom{n+m}{1}X^{n+m-1}Y,\ldots\!,
\tbinom{n+m}{n+m-1}XY^{n+m-1},\tbinom{n+m}{n+m}Y^{n+m}\right)\notag
\end{align}
Then $v=(\mathcal A\cdot w^\text{t})^\text{t}$, where the square matrix $\mathcal A=(a_{ij})$ is given
by $a_{ij}=(i-1)^{j-1}$. Hence, $\mathcal A$ is a Vandermonde matrix, which
implies that its determinant is an element of $\Q^*$. The inverse of
$\mathcal A$, together with the inverse of $\tbinom{n+m}{n}$, now
give the desired expression for $X^nY^m$.
\end{proof}

\noindent The following lemma is in the spirit of \cite{EMV}.  Here and in the subsequent lemmas $X$ represents a system of variables $\X$.

\begin{lemma} \label{sumcomp} Let $\ideaala\subset R$ be an ideal such that $\ideaala^2=(0)$. Suppose $G,H \in \ideaala[X]^n$, and define $\phi,\gamma\in\GA_n(R)$ by $\phi=X+G$,  $\gamma=X+H$ (note, that $\phi$ and $\gamma$ are indeed invertible: $\phi^{-1}=X-G$, and $\gamma^{-1}=X-H$). Then $\phi\gamma=X+G+H$.
\end{lemma}

\begin{proof}
Straightforward.
\end{proof}

\begin{lemma} \label{monomial}
Let  $a \in R$ with $a^2=0$. Let $m\in \bN^*$. Then $\omega=(\,X+aX^m\,,\,(1-maX^{m-1})Z\,)$ lies in $\EA_2(R)$.
\end{lemma}

\begin{proof}
Define $\alpha,\beta,\gamma \in \EA_2(R)$ by $\alpha=(X-aZ,Z)$, $\beta=(X,Z-X^m)$, and $\gamma=(X,Z+(X+aX^m)^m-X^m)$.  Then $\omega=\alpha\beta\alpha^{-1}\beta^{-1}\gamma$.
\end{proof}

The following will be used in the proof of Theorem \ref{modnil}.

\begin{propo} \label{theprop}
Let $\ideaala \subseteq R$ an ideal such that $\ideaala^2=(0)$.  Suppose $\phi \in \GA_n(R)$ has the form $\phi=X+H$, where $H=(H_1,\ldots,H_n) \in \ideaala[X]^n$. 
\begin{enumerate}
\item[(1)] Let $d=|J\phi|$. Letting $Z$ be a single new variable, we have
$$
(\,X+H,d^{-1}Z\,) \in \EA_{n+1}(R).
$$
Consequently, if $|J\phi| \in R^*$, then $\phi^{[1]}$ is a tame
automorphism.
\item[(2)] Suppose $|J\phi|=1$ and $R$ is a $\Q$-algebra.  Then $\phi\in\EA_n(R)$.
\end{enumerate}
\end{propo}

\begin{proof}
For (1), first note that $d=(1+\frac{\pa H_1}{\pa X_1})\cdots(1+\frac{\pa H_n}{\pa X_n})$,
so $d^{-1}=(1-\frac{\pa H_1}{\pa X_1})\cdots(1-\frac{\pa H_n}{\pa X_n})$ and
\begin{align}
(\,X+H,d^{-1}Z\,)=&\left(\,X_1+H_1,X_2,\ldots,X_n,(1-\frac{\pa H_1}{\pa X_1})\,Z\right)\notag\\
&\circ\left(\,X_1,X_2+H_2,X_3,\ldots,X_n,(1-\frac{\pa H_2}{\pa X_2})\,Z\,\right)\notag\\
&\circ\cdots\circ\left(\,X_1,\ldots,X_{n-1},X_n+H_n,(1-\frac{\pa H_n}{\pa X_n})\,Z\,\right).\notag
\end{align}
Hence, we are reduced to the case $n=1$.  

So now let $X$ represent a single variable.  For any $p(X), q(X) \in \ideaala[X]$,
$$
\left(\,X+p+q,\left(1-\frac{\pa (p+q)}{\pa
X}\right)Z\,\right)=\left(\,X+p,\left(1-\frac{\pa p}{\pa
X}\right)Z\,\right)\circ\left(\,X+q,\left(1-\frac{\pa q}{\pa X}\right)Z\,\right).
$$
This additivity allows us to assume $H$ is a monomial
$aX^m$, where $a \in \ideaala$. But this
case is precisely Lemma~\ref{monomial}.

For the proof of (2), we first, consider the case $n=2$, and for the moment we write $X,Y,g,h$ instead of $X_1,X_2,H_1,H_2$. Since $\ideaala^2=(0)$, $|J(\phi)|=1+\frac{\partial g}{\partial X}+\frac{\partial h}{\partial Y}$. Then $\frac{\partial g}{\partial X}+\frac{\partial h}{\partial Y}=0$, and since $R$ is a $\Q$-algebra, this implies that there exists a polynomial $p \in R[X,Y]$ such that $g=\frac{\partial p}{\partial Y}$ and $h=-\frac{\partial p}{\partial X}$. Using Lemma~\ref{sumcomp}, we may assume that $p=rX^nY^m$ for some $r \in A,\ n,m\ge0$ and $n+m\geq1$. With Lemma~\ref{linearcomb}, we can write $X^nY^m$ as a $\Q$-linear combination of polynomials of the form $(X+aY)^{n+m}$, with $a \in \Q$. Applying Lemma~\ref{sumcomp} again, we may assume that
$$
\phi=\left(X+kabr(X+aY)^{k-1},Y-kbr(X+aY)^{k-1}\right)\,,
$$
where $k\ge1,\ a,b \in \Q$ and $r \in R$. But then $\varphi=\alpha^{-1}\beta\alpha$, where $\alpha=(X+aY,Y)$ and $\beta=(X,Y-kbrX^{k-1})$.  Therefore $\phi\in\EA_2(R)$.

Now we turn to the general $\Q$-algebra case. For
$i=1,\ldots\!,n-1$, choose a polynomial $P_i \in
\ideaala[\X]$ such that $H_i=\frac{\partial
P_i}{\partial X_n}$. If we define $\alpha_i$ by
$$
\alpha_i=\left(X_1,\ldots\!,X_{i-1},X_i-\frac{\partial P_i}{\partial
X_n},X_{i+1},\ldots\!,X_n+\frac{\partial P_i}{\partial X_i}\right)\,,
$$
then, applying extension of scalars and appealing to the case of two variables, it
follows that $\alpha_i \in \EA_n(R)$. Furthermore,
Lemma~\ref{sumcomp} gives
$$
\alpha_1\cdots\alpha_{n-1}\phi=\left(X_1,\ldots\!,X_{n-1},X_n+\frac{\partial
P_1}{\partial X_1}+\cdots+\frac{\partial P_{n-1}}{\partial
X_{n-1}}+H_n\right)\,.
$$
As $|J(\alpha_1\cdots\alpha_{n-1}\phi)|=1$, we must have
$\frac{\partial}{\partial X_n}(\frac{\partial P_1}{\partial
X_1}+\cdots+\frac{\partial P_{n-1}}{\partial X_{n-1}}+H_n)=0$. Hence,
$\frac{\partial P_1}{\partial X_1}+\cdots+\frac{\partial
P_{n-1}}{\partial X_{n-1}}+H_n \in R[X_1,\ldots\!,X_{n-1}]$ and
$\alpha_1\cdots\alpha_{n-1}\phi \in \EA_n(R)$.  This shows $\phi\in\EA_n(R)$.
\end{proof}

The next lemma introduces an important tool: the homomorphism $\Psi_t$.  Here we write $X$ and $Z$ for sets of variables $\X$ and $Z_1,\ldots,Z_n$, and $F=F(X)$ for a vector of polynomials $(\F)\in R[\X]^n$.

\begin{propo}\label{trick} Let $t\in R$ be a non-zero-divisor.  For $\varphi\in\GA_n(R_t)$ write $\varphi$ in the form $\varphi=X+F(X)$.  The map $$X+F(X)\mapsto Z+\frac{1}{t}F(X+tZ)$$ defines a group homomorphism $\Psi_t:\GA_n(R_t)\to\GA_n(R_t[X])$.  Moreover, if $\varphi\in\GA_n(R)$ with $F\in tR[X]^n$, then $\Psi_t(\varphi)$ lies in $\GA_n(R[X])$ and is elementarily equivalent to $\varphi^{[n]}$ in $\GA_{2n}(R)$.
\end{propo}

\begin{proof} Letting $\eta=(X,Z+(1/t)X),\sigma=(X-tZ,Z)\in\GA_{2n}(R_t)$, a direct computation shows that $\Psi_t(\varphi)=\sigma\eta\varphi^{[n]}\eta^{-1}\sigma^{-1}$.  This shows that $\Psi_t$ is a group homomorphism.  However, we can also write $\Psi_t(\varphi)=\sigma\varphi^{[n]}\omega\sigma^{-1}$ where $\omega=(X,Z+(1/t)F(X))$.  If $F\in tR[X]^n$ then $\omega$ is elementary over $R$, and since $\sigma\in\EA_{2n}(R)$ it follows that $\Psi_t(\varphi)$ is elementarily equivalent over $R$ to $\varphi^{[n]}$.
\end{proof}

\begin{exam}\label{examp} We observe the effect of $\Psi_t$ in two special situations:
\begin{enumerate}
\item Let $\varepsilon=e_i(f)$, where $f(X)\in R_t[\Xmi]$.  Then
\begin{equation}\label{Psielem} \Psi_t(\varepsilon)=e_i\left(\frac{1}{t}f(X+tZ)\right)\,.
\end{equation}
\item Let $\gamma\in\GL_n(R_t)$.  Let $\mathcal I+\mathcal M$ be its matrix representation ($\mathcal I$ being the identity matrix), so that (by the slight abuse of language mentioned earlier) $\gamma=(\mathcal I+\mathcal M)X$.  We then have
\begin{equation}\label{Psilinear} \Psi_t(\gamma)=(\mathcal I+\mathcal M)Z+\frac{1}{t}\mathcal MX = (Z+\frac{1}{t}\mathcal MX)\circ\tilde\gamma
\end{equation}
where $\tilde\gamma=(\mathcal I+\mathcal M)Z$.  Note that $\tilde\gamma\in\GL_n(R_t)$ having the same matrix as $\gamma$, except in the variables $Z$ instead of $X$.
\end{enumerate}
\end{exam}

\begin{lemma}\label{sweepleft}  Let $t$ and $\Psi_t$ be as in Proposition \ref{trick}, and let $\rho\in\EA_n(R)$.  Then there exists $\tilde\rho\in\EA_n(R[X])$ and a translation $\tau\in\Tr_n(R_t[X])$ of the form $\tau=Z+(1/t)p(X)$, with $p(X)\in R[X]^n$, such that $\Psi_t(\rho)=\tau\tilde\rho$.
\end{lemma}

\begin{proof}  Write $\rho=\rho_1\cdots\rho_s$ with each $\rho_i$ being elementary.  Then $\Psi_t(\rho)=\Psi_t(\rho_1)\Psi_t(\rho_2\cdots\rho_s)$ and by induction on $s$ we have $\Psi_t(\rho_2\cdots\rho_s)=\tau'\tilde\rho'$ of the required form, taking $\tau'$ and $\tilde\rho'$ to be the identity if $s=1$.  Write 
$\rho_1=e_i(r(X))$,   where $r(X)\in R[\Xmi]$, so that, according to (\ref{Psielem}),
$$\Psi_t(\rho_1)=e_i\left(\frac{1}{t}r(X+tZ)\right)\in\EA_n(R_t[X])\,.$$
Also write $\tau'=Z+(1/t)q(X)$ with $q(X)\in R[X]^n$.  Using the $\delta$ notation introduced in \ref{delta}, we have
$\Psi_t(\rho_1)\tau'=Z+(1/t)(q(X)+r(X+q(X)+tZ)\delta_i)$.
By Taylor's expansion we can write 
$$\frac{1}{t}r(X+q(X)+tZ)=\frac{1}{t}r(X+q(X))+\tilde r(X,Z)\,,$$ 
with  $\tilde r(X,Z)\in R[X][\Zmi]$.  Then we have $\Psi_t(\rho_1)\tau'=\tau\tilde\rho_1$, where $\tau=Z+(1/t)(q(X)+r(X+q(X))\delta_i)$ and $\tilde\rho_1=e_i(\tilde r(X,Z))$.  Note that $\tau$ has the form specified by the lemma, and that $\tilde\rho_1$ is elementary over $R[X]$.  Setting $\tilde\rho=\tilde\rho_1\tilde\rho'$, we have $\Psi_t(\rho)=\tau\tilde\rho$ as desired.
\end{proof}

\begin{defi} \label{order} 
Let $t\in R$ be a non-zero-divisor.  For $a\in R_t$ we define the {\it$t$-order} of $a$ to be the smallest integer $n\ge0$ such that $t^na\in R$.  Note that the $t$-order of $a$ is zero if and only if $a\in R$.  If $\gamma$ is a matrix or vector over $R_t$ we define the $t$-order of $\gamma$ to be the maximum of the $t$-orders of its entries.
\end{defi}

\begin{rema} This definition of order might be the negative of what the reader expects.  Note that it resembles the order of a pole rather than a zero.  Also the insistence that $n$ be non-negative does not coincide with typical order functions.  However this definition will serve us well in this paper.
\end{rema}

\begin{lemma}\label{elemsweep} Let $t\in R$ be a non-zero-divisor. Let $u\in R_t^n$, $x\in R^n$, $f(X)\in R_t[\Xmi]$.  Let $T$ be an indeterminate and define $\epsilon,\sigma\in\GA_n(R_t[T,T^{-1}])$ by 
$$\epsilon=e_i\left(\frac{1}{T}f(x+TX)\right)\,,\qquad\sigma=X+\frac{1}{T}u\,.$$
Then there exist $w\in R_t^n$, $\omega\in\E_n(R_t)$, and $g(T,X)\in R_t[T][\Xmi]$ such that, letting 
$$\nu=X+\frac{1}{T}w\,,\qquad\xi=e_i(Tg(T,X))\,,$$
we have $\epsilon\sigma=\nu\omega\xi$.  Moreover the $t$-orders of $\omega$, $w$, and $g$ are $\le m$ where $m$ is a number depending only on the degree of $f$ and the $t$-orders of $f$ and $u$.\footnote{This sentence is not quite precise.  Literally we mean, more strongly, that given integers $d,r,s\ge0$ there exists an integer $m=m(d,r,s)\ge0$ such that given \underbar{any} $u$ and $f$ as in the lemma with $\deg(f)\le d$, $f$ having $t$-order $\le r$ and $u$ having $t$-order $\le s$, then the resulting $\omega$, $w$, and $g$ will have $t$-order $\le m$.  This abuse will be repeated in Lemma \ref{shorten}.}
\end{lemma}

\begin{proof}  A quick computation shows
\begin{equation}\label{a}\epsilon\sigma=e_i\left(\frac{1}{T}f(x+TX)\right)\circ \left(X+\frac{1}{T}u\right)
=\left(X+\frac{1}{T}u\right)\circ e_i\left(\frac{1}{T}f(x+u+TX)\right)\,.
\end{equation}
Using Taylor's expansion we see that $(1/T)f(x+u+TX)$ can be written as $(1/T)f(x+u)+\sum_{j=1}^n \frac{\partial f}{\partial X_j}(x+u)X_j+Tg(T,X)$ with $g$ as prescribed in the lemma.  (Note that the $i^\text{th}$ summand in the middle summation is zero.)  Therefore (\ref{a}) gives
\begin{align}\epsilon\sigma&=\left(X+\frac{1}{T}u\right)\circ\left(X+\frac{1}{T}f(x+u)\delta_i\right)\circ e_i\left(\sum_{j=1}^n \frac{\partial f}{\partial X_j}(x+u)X_j\right)\circ e_i(Tg(T,X))\notag\\
&=\left(X+\frac{1}{T}\left(u+f(x+u)\delta_i\right)\right)\circ e_i\left(\sum_{j=1}^n \frac{\partial f}{\partial X_j}(x+u)X_j\right)\circ e_i(Tg(T,X))\,.\label{b}
\end{align}
Letting $w=u+f(x+u)\delta_i$ and letting $\nu$, $\omega$, and $\xi$ be the three respective automorphisms in (\ref{b}), we have $\epsilon\sigma=\nu\omega\xi$ as desired.  Notice that the assertion about the $t$-orders is apparent from the definitions of $w$ and $\omega$.
\end{proof}

Lemmas \ref{linconjelem} and \ref{elemconjelem} give commutator formulas that will be needed for our results involving stable tameness and localization. 

\begin{lemma}[First Commutator Formula]\label{linconjelem}  Let $\alpha\in\GL_n(R)$ and let $\varepsilon=e_i(bf(X))$ for some $i\in\{1,\ldots,n\}$, $b\in R$, $f(X)\in R[\Xmi]$.  Let $\mathcal A$ denote the matrix of $\alpha$ and let $a$ be the $i^\text{th}$ column of $\mathcal A$. then
$$\left(\alpha\varepsilon\alpha^{-1}\right)^{[1]}=\kappa\nu\kappa^{-1}\nu^{-1}$$
where
$$\kappa=(X+a^\text{t}bY,Y),\qquad\nu=(X,Y+f(\mathcal A^{-1}X))\,,$$
$Y$ being the variable representing the added dimension.
\end{lemma}


\begin{proof}
Let $\kappa_0=(X+bY\delta_i,Y)$ and $\nu_0=(X,Y+f(X))$. Then
\begin{multline*}
\left.\begin{array}{l}
\kappa_0\nu_0=(X+bY\delta_i+bf(X)\delta_i,Y+f(X))\\
\kappa_0^{-1}\nu_0^{-1}=(X-bY\delta_i+bf(X)\delta_i,Y-f(X))
\end{array}\right\}\\
\Longrightarrow\
\kappa_0\nu_0\kappa_0^{-1}\nu_0^{-1}=(X+bf(X)\delta_i,Y)=\varepsilon^{[1]}
\end{multline*}
(the latter resulting from the fact that
$f(X-\delta_ibY+\delta_ibf(X))=f(X)$ since $f \in R[X,\hat{i}]$).
Also,
$$
\alpha^{[1]}\kappa_0(\alpha^{[1]})^{-1}=(\mathcal A(\mathcal A^{-1}X)+\mathcal A(bY\delta_i),Y)=\kappa
$$
$$
\alpha^{[1]}\nu_0(\alpha^{[1]})^{-1}=(\mathcal A(\mathcal A^{-1}X),Y+f(\mathcal A^{-1}X))=\nu
$$
whence
$$
\kappa\nu\kappa^{-1}\nu^{-1}  = \alpha^{[1]}(\kappa_0\nu_0\kappa_0^{-1}\nu_0^{-1})(\alpha^{[1]})^{-1}=\alpha^{[1]}\varepsilon^{[1]}(\alpha^{[1]})^{-1}=(\alpha\varepsilon\alpha^{-1})^{[1]}\,,
$$
and we are done.
\end{proof}

Lemma \ref{linconjelem} has the following two corollaries which are interesting in themselves:

\begin{corol}  The group $\EA_{\infty}(R)$ is a normal subgroup of $\T_{\infty}(R)$.
\end{corol}
\begin{proof}  It is clear that $\T_{\infty}(R)=\langle\,\GL_{\infty}(R),\EA_{\infty}(R)\,\rangle$.  Thus we have only to show that $\GL_{\infty}(R)$ is in the normalizer of $\EA_{\infty}(R)$, which is immediate from the fact that both $\kappa$ and $\nu$ (from the lemma, setting $b=1$) lie in $\EA_{n+1}(R)$.
\end{proof}

\begin{ques} Is $\EA_{\infty}(R)$ a normal subgroup of $\GA_{\infty}(R)$?
\end{ques}

\begin{corol}\label{linconjelemcor} Let $t\in R$ be a non-zero-divisor.  Let $m\ge0$ be an integer and let $\alpha\in\GL_n(R_t)$ be such that the $t$-orders of $\alpha$ and $\alpha^{-1}$ are $\le m$.  Let $\varepsilon=e_i(g(X))\in\GA_n(R)$ with $g(X)\in t^{m+dm}R[\Xmi]$, where $d=\degr\,g(X)$.  Then $\alpha\varepsilon\alpha^{-1}\in\GA_n(R)$ and $(\alpha\varepsilon\alpha^{-1})^{[1]}$ lies in $\EA_{n+1}(R)$.
\end{corol}

\begin{proof}  Writing $g=t^mf$ where $f\in t^{dm}R[\Xmi]$, we apply Lemma \ref{linconjelem} with $b=t^m$ (and with $R_t$ in the role of the lemma's $R$).  Our hypotheses imply that $a^\text{t}b\in R^n$ and $f(\mathcal A^{-1}X)\in R[\Xmi]$, so the conclusion follows.
\end{proof}

The next lemma is a statement about two dimensional automorphisms, for which $X_1,X_2$ will be our dimension variables.

\begin{lemma}[Second Commutator Formula]\label{elemconjelem}  Let $\psi,\varepsilon\in\EA_2(R)$ be elementary of the form $\psi=e_1(f)$, $\varepsilon=e_2(bg)$, where $f\in R[X_2]$, $g\in R[X_1]$, $b\in R$.  Then
$$\left(\psi\varepsilon\psi^{-1}\right)^{[1]}=\gamma\omega\gamma^{-1}\omega^{-1}$$
where
$$\gamma=(X_1+f(X_2+bY)-f(X_2),X_2+bY,Y)\qquad\omega=(X_1,X_2,Y+g(X_1-f(X_2)))$$
with $Y$ representing the added dimension.
\end{lemma}


\begin{proof} 
Let $\gamma_0=(X_1,X_2+bY,Y)$ and $\omega_0=(X_1,X_2,Y+g(X_1))$.
Then
\begin{multline*}
\left.\begin{array}{l}
\gamma_0\omega_0=(X_1,X_2+bY+bg(X_1),Y+g(X_1))\\
\gamma_0^{-1}\omega_0^{-1}=(X_1,X_2-bY+bg(X_1),Y-g(X_1))
\end{array}\right\}\Longrightarrow\
\gamma_0\omega_0\gamma_0^{-1}\omega_0^{-1}\\=(X_1,X_2+bg(X_1),Y)=\varepsilon^{[1]}
\end{multline*}
Also, $\psi^{[1]}\gamma_0(\psi^{[1]})^{-1}=\gamma$ and
$\psi^{[1]}\omega_0(\psi^{[1]})^{-1}=\omega$, whence
$$
\gamma\omega\gamma^{-1}\omega^{-1} = \psi^{[1]}(\gamma_0\omega_0\gamma_0^{-1}\omega_0^{-1})(\psi^{[1]})^{-1}=\psi^{[1]}\varepsilon^{[1]}(\psi^{[1]})^{-1}=(\psi\varepsilon\psi^{-1})^{[1]}\,,
$$
which completes the proof.
\end{proof}

We have as a consequence\footnote{A non-stable statement of this kind appears in \cite{Connell} (Lemma 2.2) and in \cite{vdE} (Proposition 5.2.3).  In \cite{Connell} the definition of $\EA_n(R)$ (Definition 2.1) is slightly more restrictive, coinciding with ours when $R$ is a $\Q$-algebra.  Some other interesting facts are proved in \cite{Connell}.  For example it is shown (Theorem 2.7) that the group $[\GA_\infty(R),\GA_\infty(R)]$ is perfect and coincides with the normal closure of $\EA_\infty(R)$ in $\GA_\infty(R)$ (and here the difference in the definitions of $\EA_n(R)$ becomes moot).  An open question, posed in \cite{Connell}, is the following:  Is $\EA_\infty(R)=[\GA_\infty(R),\GA_\infty(R)]\,$?}:

\begin{corol} The group $\EA_\infty(R)$ is perfect, i.e., $\EA_\infty(R)=[\,\EA_\infty(R),\EA_\infty(R)\,]$.
\end{corol}

\begin{proof} It is easily seen that the automorphism $\gamma$ (as well as $\omega$) lies in $\EA_3(R)$.  Also, an $n$-dimensional elementary automorphism can be viewed as 2-dimensional after applying extension of scalars.  Thus the corollary follows from Lemma \ref{elemconjelem} taking $\psi=\id$ and $b=1$.
\end{proof}

The following lemma, which is a bit technical and also very subtle, will play a critical role in the main results.  Here $X=\X$ and $Z=\Z$ are systems of variables.

\begin{lemma}\label{shorten}  Let $\alpha\in\GL_n(R_t)$, and let $\varepsilon=e_i(f)$ with $f(X)\in R_t[\Xmi]$.  Let $\gamma\in\GL_n(R_t[X])$ and let $p(X)\in R_t[X]^n$.  Let $\tau=Z+(1/t^N)p(X)$.  Then there exist $\tilde\gamma\in\GL_n(R_t[X])$ and $\tilde p(X)\in R_t[X]^n$ such that for $N$ sufficiently large, there exists $\zeta\in\EA_{n+1}(R[X])$ (depending on $N$) such that 
\begin{equation}\label{convert}\left(\Psi_{t^N}(\alpha)\Psi_{t^N}(\varepsilon)\tau\gamma\right)^{[1]}=(\tilde\tau\tilde\gamma)^{[1]}\zeta\end{equation} 
where $\tilde\tau=Z+(1/t^N)\tilde p(X)$.  The required magnitude of $N$ is dependent only on the $t$-orders of $\alpha$, $f(X)$, $\gamma$, and $p(X)$, and the degree of $f(X)$.  Also, the $t$-orders of $\tilde\gamma$ and $\tilde p(X)$ can be bounded by a function depending only on these parameters as well.
\end{lemma}

\begin{rema}\label{important} It is crucial that $\tilde p(X)$ and $\tilde\gamma$ of Lemma \ref{shorten}, and the bound on their $t$-orders, depend \underbar{only} on $\alpha$, $f(X)$, $p(X)$ and $\gamma$, and \underbar{not} on $N$.  Observe in the proof below that they are specified before $N$ is chosen.  Only $\zeta$ depends on $N$.
\end{rema}

\begin{proof}  Letting $T$ be a new variable, we define the following elements of the group $\GA_n(R_t[T,T^{-1},X])$: 
$$\epsilon=\epsilon(T)=e_i\left(\frac{1}{T}f(X+TZ)\right)\,,\qquad \sigma=\sigma(T)=Z+\frac{1}{T}p(X)\,.$$
We apply Lemma \ref{elemsweep}, with $R[X]$ and $Z$ playing the roles of the lemma's $R$ and $X$, to get $\epsilon\sigma=\nu\omega\xi$, with $\omega\in\E_n(R_t[X])$ and
$$\nu=Z+\frac{1}{T}w(X)\,,\qquad\xi=e_i(Tg(T,Z))\,,$$ where $w(X)\in R_t[X]^n$, $g(T,Z)\in R_t[X][T,\Zmi]$.  Since the $Z$-degree of $f(X+TZ)$ is the same as the $X$-degree of $f(X)$, Lemma \ref{elemsweep} also tells us that the $t$-orders of $\omega$, $w(X)$, and $g(T,Z)$  are bounded by a function of the $X$-degree of $f(X)$ and the $t$-orders of $f(X)$ and $p(X)$.

Compose on the right with $\gamma$ to get $\epsilon\sigma\gamma=\nu\omega\xi\gamma=\nu\omega\gamma\gamma^{-1}\xi\gamma$.

Now write the matrix of $\alpha$ as $\mathcal I+\mathcal M$ ($\mathcal I$ being the identity matrix) and define $$\beta=\beta(T)=(\mathcal I+\mathcal M)Z+\frac{1}{T}\mathcal MX\in\Af_n(R_t[T,T^{-1},X])\,.$$
Note that $\beta\nu=\nu'\beta'$  where $\nu'=Z+(1/T)[(\mathcal I+\mathcal M)w(X)+\mathcal MX]$ and $\beta'=(\mathcal I+\mathcal M)Z$.  Letting
\begin{equation}\label{tildep}
\tilde p(X)=(\mathcal I+\mathcal M)w(X)+\mathcal MX
\end{equation} 
we have $\nu'=Z+(1/T)\tilde p(X)$.  It is clear from (\ref{tildep}) that the $t$-order of $\tilde p(X)$ is bounded by a function of the $t$-orders of $\alpha$ and $w(X)$, and we have already observed that the latter $t$-order is bounded by a function of the $X$-degree of $f(X)$ and the $t$-orders of $f(X)$ and $p(X)$.

Thus we have (and here we indicate precisely which automorphisms involve $T$): 
\begin{equation}\label{unspecial}\beta(T)\epsilon(T)\sigma(T)\gamma=\nu'(T)\beta'\omega\gamma\gamma^{-1}\xi(T)\gamma\,.\end{equation}  
We now observe that $\beta(t^N)=\Psi_{t^N}(\alpha)$, $\epsilon(t^N)=\Psi_{t^N}(\varepsilon)$ (see Example \ref{examp}), and $\sigma(t^N)=\tau$.  Therefore setting $T=t^N$ in (\ref{unspecial}) gives $\Psi_{t^N}(\alpha)\Psi_{t^N}(\varepsilon)\tau\gamma=\nu'(t^N)\beta'\omega\gamma\gamma^{-1}\xi(t^N)\gamma$.  Setting $\tilde\tau=\nu'(t^N)$ and 
\begin{equation}\label{tildegamma}
\tilde\gamma=\beta'\omega\gamma=((\mathcal I+\mathcal M)Z)\circ\omega\circ\gamma
\end{equation} 
(which lies in $\GL_n(R_t[X])\,$), we obtain
\begin{equation}\label{almost}\Psi_{t^N}(\alpha)\Psi_{t^N}(\varepsilon)\tau\gamma=\tilde\tau\tilde\gamma\gamma^{-1}\xi(t^N)\gamma\,.
\end{equation}
It is apparent from (\ref{tildegamma}) and observations made earlier that the $t$-order of $\tilde\gamma$ is bounded by a function of the stated parameters.

Finally, we apply Corollary \ref{linconjelemcor} to $\gamma^{-1}\xi(t^N)\gamma$, with $R[X]$ in the place of $R$.  Since $\xi=e_i(Tg(T,Z))$ it is clear that a sufficiently large choice of $N$ will make $\xi(t^N)$ meet the hypothesis of the corollary, so that, setting $\zeta=(\gamma^{-1}\xi(t^N)\gamma)^{[1]}$, equation (\ref{convert}) follows from (\ref{almost}) and $\zeta\in\EA_{n+1}(R[X])$ as desired. 

Note that the required magnitude of $N$ depends on the $t$-order of $g(T,Z)$ which was provided by Lemma \ref{elemsweep}, and, accordingly, depends only on the prescribed parameters.
\end{proof}

\begin{propo}\label{elem}  If $R$ is a ring for which $\SL_n(R)=E_n(R)$, then $\T_n(R)\cap\SA_n(R)=\EA_n(R)$.  The hypothesis holds when $R$ is a local ring.
\end{propo}

\begin{proof}  The first statement follows easily from these two facts:  (1) Any element of $\GL_n$ can be written as a product of a diagonal element times an element of  $\SL_n$, and (2) conjugating an elementary automorphism by a diagonal automorphism yields an elementary automorphism.

For the second statement one can use elementary operations to diagonalize (using the fact that $R$ is local) then use the fact (true for any ring $R$) that 
$$
\begin{pmatrix} u & 0\\ 0&u^{-1}\end{pmatrix}\in\E_2(R)
$$
when $u\in R^*$.
\end{proof}

In the lemma below, note that the element $t\in R$ is allowed to be a zero-divisor.

\begin{lemma}\label{loclem}  Let $t\in R$, and let $Z$ be an indeterminate.  Let $\psi,\phi\in\GA_n(R[Z])$, both $Z$-vanishing, such that $\psi_t=\phi_t$ in $\GA_n(R_t[Z])$.  Then for $N$ sufficiently large, $\psi(t^NZ)=\phi(t^NZ)$.
\end{lemma}

\begin{proof}  This is an easy consequence of the fact that if $a\in R$ goes to zero in $R_t$, then $t^Na=0$ in $R$ for some $N$.
\end{proof}

\begin{lemma}\label{linelemconj} Let $Z, T$ be indeterminates.  Let $\psi\in\GA_n^0(R)$ be such that either $\psi\in\GL_n(R)$ or $\psi $ is elementary.  Let $\varepsilon(Z)\in\EA_n^0(R[Z])$ be elementary and $Z$-vanishing (see Definition (\ref{vanishing})).  Then $(\psi\varepsilon(TZ) \psi ^{-1})^{[1]}$ is a finite product of $Z$-vanishing and $T$-vanishing elementary origin preserving automorphisms in $\EA_{n+1}^0(R[Z,T])$.
\end{lemma}

\begin{proof}  We can write $\varepsilon(Z)=e_i(Zg(Z,X))$ with $g(Z,X)\in R[Z,\Xmi]$ for some $i\in\{1,\ldots,n\}$ and $g(Z,0)=0$.

First let us assume $\psi\in\GL_n(R)$.  Letting $\mathcal A$ denote the matrix of $\psi$ and $a$ the $i^{\text{th}}$ column of $\mathcal A$, we have, according to the First Commutator Formula, Lemma \ref{linconjelem} (with $T$ in the role of the lemma's $b$ and $R[Z,T]$ in the role of the lemma's $R$), $$
(\psi\varepsilon(TZ) \psi ^{-1})^{[1]}=\kappa\nu\kappa^{-1}\nu^{-1}
$$
where $\kappa=(X+a^\text{t}TY,Y)$, $\nu=(X,Y+Zg(TZ,\mathcal A^{-1}X))$.  (Here $Y$ represents the added dimension.)  Clearly $\nu$ is $Z$-vanishing and origin preserving, and $\kappa$ is the product of $T$-vanishing, origin preserving elementary automorphisms.

Now assume $\psi $ is elementary and origin preserving.  If $\psi $ and $\varepsilon(Z)$ are elementary in the same position, they commute, and hence $\psi\varepsilon(TZ) \psi ^{-1}=\varepsilon(TZ)$, which is $Z$-vanishing (and $T$-vanishing as well).  Otherwise all but two of the variables, say $X_3,\ldots,X_n$, are fixed by both, so we want to treat them as scalars and write $\psi =(X_1+f,X_2)$, $\varepsilon=(X_1,X_2+Zg)$.  The only problem is that $\phi$ and $\varepsilon$ may not be origin preserving as two-dimensional automorphisms, so let us record that $f\in R[X_2,\ldots,X_n]$, $g\in R[Z,X_1,X_3,\ldots,X_n]$ and that $f(0)=g(Z,0)=0$.  Then, again letting $Y$ be the variable representing the added dimension (and suppressing $X_3,\ldots,X_n$), the result follows from the Second Commutator Formula, Lemma \ref{elemconjelem}, which says that
$$(\psi\varepsilon(TZ)\psi ^{-1})^{[1]}=\gamma\omega\gamma^{-1}\omega^{-1}$$ where
\begin{align}\gamma&=(X_1+f(X_2+TY)-f(X_2),X_2+TY,Y)\notag\\
\omega&=(X_1,X_2,Y+Zg(TZ,X_1-f(X_2)))\notag\end{align}
(here $T$ plays the role of the lemma's $b$ and $Zg(TZ,X_1)$ is in the role of the lemma's $g(X_1)$) and the fact that 
\begin{equation}\label{both}\gamma=(X_1+T(T^{-1}(f(X_2)-f(X_2-TY))),X_2,Y)\circ(X_1,X_2+TY,Y)\,;
\end{equation}  
namely, we observe that $\omega$ is $Z$-vanishing and that the two elementary automorphisms in the factorization (\ref{both}) are $T$-vanishing, and that all three of these are origin preserving when considered as $(n+1)$-dimensional automorphisms in the full set of variables $\X,Y$.
\end{proof}

Crucial to our results will be the following result of Suslin, which is a reformulation of \cite{Sus}, Corollary~6.5. This will be used in the proof of Theorem \ref{locmod}.\footnote{The first statement of Theorem \ref{locmod} only needs the fact that $\GL_\infty(R^{[m]})=\left\langle\E_\infty(R^{[m]}),\GL_\infty(R)\right\rangle$, which is just the assertion that the map $\text{K}_1(R)\to\text{K}_1(R^{[m]})$ is an isomorphism.  This was proved by Bass, Heller, and Swan in \cite{BHS}.}

\begin{theo}[Suslin]\label{suslin} Let $R$ be a regular ring.  Then $$\GL_n(R^{[m]})=\left\langle\E_n(R^{[m]}),\GL_n(R)\right\rangle$$ for $n\ge\text{max}\,(3,2+\text{dim}\, R)$.
\end{theo}

\begin{rema}
For $R$ a polynomial ring over a field, an algorithmic proof of this theorem has been given in \cite{P-W}, making the proof of Theorem \ref{locmod} constructive in this case.
\end{rema}

\section{The Main Results}\label{mainresults}

The following theorem will be an important component in the proof of Theorem \ref{main}, but it is also of interest in its own right.  Some ideas from \cite{EMV} are employed.

\begin{theo}\label{modnil} Let $A$ be a ring, $I$ an ideal contained in the nilradical of $A$, $\bar A=A/I$. Let $\varphi \in \SA_n(A)$. If $\bar \varphi \in \EA_n(\bar A)$, then $\varphi$ is stably a composition of elementary automorphisms, i.e., $\varphi^{[m]} \in \EA_{n+m}(A)$ for some $m\geq0$. If $A$ is a $\Q$-algebra, then we have more strongly that $\varphi \in \EA_n(A)$.
\end{theo}

\begin{proof}
Since the assumption that $\bar \varphi \in \EA_n(\bar A)$ can be expressed using only finitely many coefficients in the ideal $I$, we
may assume that $I$ is finitely generated. Hence it is a nilpotent ideal, say $I^D=(0)$ for some integer $D\ge1$. We will prove by
induction on $D$ that $\varphi$ is a product of elementary automorphisms.

In the case $D=1$ we have $I=0$, so there is nothing to prove. So now let $D\geq2$ and let $\tilde A=A/I^{D-1}$ and $\tilde I=I/I^{D-1}$. Since $\tilde{\varphi}\in\SA_n(\tilde A)$, the induction hypothesis (applied to the ring $\tilde A$ and its ideal $\tilde I$) says that, for
some $r\ge0$, $\bar\varphi^{[r]}$ is a composition of elementary automorphisms, i.e., $\tilde\varphi^{[r]}\in\EA_{n+r}(\tilde A)$. We can
lift each of these elementary automorphisms to elementary automorphisms over $A$ (see \ref{lift}) to produce $\varepsilon\in\EA_{n+r}(A)$ such that $\varepsilon^{-1}\varphi^{[r]}=X+H$, where $X=(X_1,\ldots,X_{n+r})$ and $H=(H_1,\ldots,H_{n+r}) \in I^{D-1}[X]^{n+r}$.  

Let $\rho=\varepsilon^{-1}\varphi^{[r]}$.  Since $(I^{D-1})^2=0$ we can apply (1) of Proposition \ref{theprop}, which, since $|J(\rho)|=1$, asserts that $\rho^{[1]}\in\EA_{n+r+1}(A)$.  Hence $\varphi^{[r+1]}\in\EA_{n+r+1}(A)$ as well, and the first assertion is proved.

In the case $A$ is a $\Q$-algebra, we proceed as above, but the induction hypothesis gives that $\tilde\varphi$ itself is a
composition of elementary automorphisms, i.e., $\tilde{\varphi}\in\EA_n(\tilde A)$.  As before, we lift each of these to elementary automorphisms over $A$, compose them to form $\varepsilon\in\EA_n(A)$, and replace $\varphi$ by $\varepsilon^{-1}\varphi$.  We can thereby assume that $\varphi=(X_1+H_1,\ldots\!,X_n+H_n)$ with $H_1,\ldots\!,H_n\in I^{D-1}[X]$.  The conclusion $\varphi\in\EA_n(A)$ now follows from (2) of Proposition \ref{theprop}.

\end{proof}

\begin{rema}\label{charp}  A close look at the inductive argument in the above proof for the general case shows that the number of new dimensions needed is $D-1$, where $D$ is the smallest integer for which $I^D=0$.  However, the procedure could have been made more efficient had we taken $\bar A=A/I^E$ and $\bar I=I/I^E$ where $E$ is the round-up of $D/2$, which is all that is needed to insure $(\bar I^E)^2=0$.  Using this method the number of new variables needed would be approximately $\log_2 D$.
\end{rema}

Theorem \ref{modnil} has the following interesting corollary, which can be viewed as a generalization of Jung's Theorem:

\begin{theo} \label{dimzerotame}
Let $A$ be an Artinian ring and let $\varphi\in\SA_2(A)$. Then $\varphi$ is stably a composition of elementary automorphisms, i.e., $\varphi^{[m]}\in\EA_{2+m}(A)$ for some $m\ge0$.  If $A$ is a $\Q$-algebra, then we have more strongly that $\varphi\in\EA_2(A)$.
\end{theo}

\begin{rema}  Co-author Joost Berson has shown that the conclusion $\varphi\in\EA_2(A)$ does not hold if the hypothesis ``$A$ is a $\Q$-algebra" is removed.  He produces a counterexample when $A=\mathbb F_p[T]/(T^2)$, $p$ a prime.  This result appears in \cite{Ber} as Theorem 5.1.
\end{rema}

\begin{proof}  We apply the Theorem \ref{modnil} taking $I$ to be the nilradical of $A$.  In this case $A/I$ is a product of fields so the hypotheses is met by virtue of \ref{prod} and the Jung-Van der Kulk Theorem (Theorem \ref{JvdK}).  (It is an easy consequence of the latter that $\SA_2(k)=\EA_2(k)$ for $k$ a field.)
\end{proof}



\begin{theo} \label{locmod} Let $R$ be a regular domain, $t\in R-\{0\}$, and $\varphi\in\GA_n(R)$.  If $\varphi\in\T_n(R_t)$ and $\bar\varphi\in\EA_n(R/tR)$, then  $\varphi$ is stably tame.  If, more strongly, $\bar\varphi\in\EA_n(R/t^NR)$ for $N$ sufficiently large, then $\varphi^{[\ell]}\in\T_{n+\ell}(R)$, where $\ell=\text{max}\,(2+\text{dim}\, R, n+1)$.
\end{theo}



\begin{proof}
Let $\varphi\in \GA_n(R)$ be as in the theorem. Letting $u=|J\varphi|$ and $\delta=(u^{-1}X_1,X_2,\ldots,X_n)$, we can replace $\varphi$ by $\delta^{-1}\varphi$ to arrange that $\varphi\in\SA_n(R)$.  Note that $\bar\varphi$ remains unchanged, since $\bar\delta=\id$.  With this preparation we note that the assumption $\bar\varphi\in\EA_n(R/tR)$ implies that $\bar\varphi\in\EA_\infty(R/t^NR)$ for $N\ge1$ (see \ref{dirlim} for notation).  This is a straightforward application of Theorem \ref{modnil}.  (Note:  We acknowledge the double usage of the symbol $\bar\varphi$, which now used for the image of $\varphi$ in $\GA_\infty(R/t^NR)$; the latter usage continues below.)

Since $\varphi\in\T_n(R_t)$ we can write $\varphi=\alpha_1\varepsilon_1\cdots\alpha_r\varepsilon_r$ for some $r$, where the $\varepsilon_1,\ldots,\varepsilon_r$ are elementary over $R_t$ and $\alpha_1,\ldots,\alpha_r\in\GL_n(R_t)$.  Write $\varepsilon_i=e_{j_i}(f_i)$ with $f_i\in R_t[X,\hat j_i]$.

Now choose a large $N$ and let $\bar R=R/t^NR$.  The size of $N$ will be determined later, but we will see that it depends only on the automorphisms $\alpha_1,\ldots,\alpha_r,\varepsilon_1,\ldots,\varepsilon_r$. We have arranged that $\bar\varphi\in\EA_\infty(\bar R)$, so after replacing $n$ by $n+p$ for some $p\ge0$ (depending on $N$), then by \ref{lift} we can lift $\bar\varphi^{-1}$ to an element $\rho\in\EA_n(R)$.  Letting $\phi=\varphi\rho$ we have $\bar\phi=\id$, and $\phi$ is elementarily equivalent to $\varphi$.  Now we apply the map $\Psi=\Psi_{t^N}$ of Proposition \ref{trick}, and note that, according to said proposition, $\Psi(\phi)$ is elementarily equivalent to $\phi^{[n]}$ in $\GA_{2n}(R)$.  Furthermore, since $\Psi$ is a homomorphism, we have $$\Psi(\phi)=\Psi(\alpha_1)\Psi(\varepsilon_1)\cdots\Psi(\alpha_r)\Psi(\varepsilon_r)\Psi(\rho)\,.$$   We note that, by extension of scalars, this factorization occurs in $\GA_n(R_t[X])$, where $X=(\X)$.  We will write $Z=(Z_1,\ldots,Z_n)$ for the new variables that have been introduced.

Our first step is to apply Lemma \ref{sweepleft} to get $\Psi(\rho)=\tau\tilde\rho$, with $\tilde\rho\in\EA_n(R[X])$ and $\tau\in\Tr_n(R_t[X])$ of the form $\tau=Z+(1/t^N)p(X)$, with $p(X)\in R[X]^n$.  We can now see that $\Psi(\phi)$ is elementarily equivalent to
\begin{equation}\label{r}\psi_r=\Psi(\alpha_1)\Psi(\varepsilon_1)\cdots\Psi(\alpha_r)\Psi(\varepsilon_r)\tau\,.\end{equation}  We now apply Lemma \ref{shorten} to the last three factors $\Psi(\alpha_r)\Psi(\varepsilon_r)\tau$ of (\ref{r}), with $\gamma=\id$.  This tells us that, with $N$ sufficiently large, $\psi_r^{[1]}$ is elementarily equivalent in $\GA_{n+1}(R[X])$ to $\psi_{r-1}^{[1]}$, where
\begin{equation}\psi_{r-1}=\Psi(\alpha_1)\Psi(\varepsilon_1)\cdots\Psi(\alpha_{r-1})\Psi(\varepsilon_{r-1})\tau_r\gamma_r\end{equation}
with $\tau_r=Z+(1/t^N)p_r(X)$ where $p_r(X)\in R_t[X]^n$, $\gamma_r\in\GL_n(R_t[X])$.  According to Lemma \ref{shorten}, evoking Example \ref{examp}, the required magnitude of $N$ depends only on the degree of $f_r(X)$ and the $t$-orders of $\alpha_r$ and $f_r$, since we know the $t$-order of $p(X)$ is zero\footnote{Of course $p(X)$ depends on $N$, but its $t$-order, zero, does not, and that is all that is needed for this step of the proof and to bound the $t$-orders of the polynomials $p_r(X)$ that appear subsequently in the proof.}.  Also the $t$-orders of $p_r(X)$ and $\gamma_r$ are bounded by a function of these same parameters.  Thus a sufficiently large choice of $N$ will suffice to apply Lemma \ref{shorten} to $\Psi(\alpha_{r-1})\Psi(\varepsilon_{r-1})\tau_r\gamma_r$ as well, since the magnitude of $N$, as well as $t$-orders of the resulting $p_{r-1}(X)$ and $\gamma_{r-1}$, will depend on the degree of $f_{r-1}(X)$ and the $t$-orders of $\alpha_{r-1}$, $f_{r-1}(X)$, $p_r(X)$, and $\gamma_r$.

Thus we continue to apply Lemma \ref{shorten} to conclude that $\Psi(\phi)$ is elementarily equivalent to $\tau_1\gamma_1$ with $\tau_1,\gamma_1$ as in the lemma.  A careful look at the hypothesis of Lemma \ref{shorten} reveals that at the beginning $N$ could be chosen large enough to suffice for each of these applications just by knowing $t$-orders of $\alpha_1,\ldots,\alpha_r,f_1,\ldots,f_r$ (see Remark \ref{important}).  Observe that the replacement of $n$ by $n+p$ was innocent, since it did not affect these $t$-orders.

We have shown that $\psi_r^{[1]}$, hence $\varphi^{[n+1]}$, is elementarily equivalent over $R$ to $(\tau_1\gamma_1)^{[1]}$, with $\tau_1\gamma_1$ lying in $\Af_n(R_t[X])\cap\GA_n(R[X])=\Af_n(R[X,Y])$.  In particular, $\tau_1$ is a translation over $R[X]$, and therefore $\varphi$ is elementarily equivalent to $\gamma_1\in\GL_n(R[X])$.  We may now appeal to Theorem \ref{suslin}, from which it follows that $\gamma_1$ becomes tame when $n\ge\max(3,2+\text{dim}\, R)$, establishing the first assertion.

For the second statement, note that under the stronger assumption $\bar\varphi\in\EA_n(R/t^NR)$,  the enlargement of $n$ was not required and we have needed no more than $\text{max}\,(3,2+\text{dim}\, R, n+1)$ added variables.  The 3, however, is redundant; for if $\text{dim}\, R=0$ then $R$ is a field and the assertion holds with $\ell=0$.  This concludes the proof.
\end{proof}

This tool furnishes an immediate proof of:

\begin{theo}\label{submain} Let $R$ be a Dedekind domain, and let $\varphi\in\GA_2(R)$.  Then $\varphi$ is stably tame.  If $R$ is also a $\Q$-algebra, then, $\varphi$ becomes tame with the addition of three more dimensions, i.e., $\GA_2(R)\subset\T_5(R)$.
\end{theo}
\begin{rema}  The more general assertion of Theorem \ref{submain} is contained in Theorem \ref{main} below.  However the sharper statement for Dedekind $\Q$-algebras is not.
\end{rema}
\begin{proof} We may assume $\varphi\in\SA_2(R)$.  By the Jung-Van der Kulk Theorem $\varphi$ becomes tame when we make the base change from $R$ to its field of fractions, so clearly $\varphi_t\in\T_2(R_t)$ for a well-chosen $t\in R$, $t\ne 0$.  For $N\ge1$ $\bar R=R/t^NR$ is an Artinian ring, so according to Theorem \ref{dimzerotame}, $\bar\varphi^{[m]}\in\EA_{2+m}(\bar R)$ for some $m\ge0$, and if $R$ is a $\Q$-algebra we can take $m=0$.  Now we apply Theorem \ref{locmod}.  In the case $R$ is a $\Q$-algebra the stronger hypothesis of Theorem \ref{locmod} holds, with $n=2$, so the $\ell$ of Theorem \ref{locmod} is $3$.  Thus the proof is complete.
\end{proof}

\begin{rema}\label{eucl}  If $R$ is a Euclidean domain we have $\SL_n(R)=\E_n(R)$ for all $n\ge1$.  (Caution: Not all Dedekind domains -- in fact, not all PIDs -- have this property.  See \cite{Gray}.)   From this it easily follows that $\GL_n(R)=\langle\,\E_n(R),\Di_n(R)\,\rangle$.  Since $\E_n(R)\subset\EA_n(R)$ we can also conclude that $T_n(R)=\langle\,\EA_n(R),\Di_n(R)\,\rangle$.   Taking $R=k[T]$, $k$ a field, this observation together with Theorem \ref{submain} imply that elements of $\GA_2(k[T])$, viewed as automorphism over $k$ by restriction of scalars, are stably tame \underbar{over $k$}.  If $k$ has characteristic zero, we have, more strongly, $\GA_2(k[T])\subset\T_6(k)$.
\end{rema}

The following derives immediately from Remark \ref{eucl}:

\begin{corol} \label{cor} Let $k$ be a field and let $W$ be the subgroup of $\GA_3(k)$ generated by all automorphisms which fix one coordinate.  Then all elements of $W$ are stably tame.  If $k$ has characteristic zero we have, more precisely, $W\subset\T_6(k)$.
\end{corol}

Note that $W$ properly contains the tame subgroup $\T_3(k)$, as Nagata's example lies in $W$ but is not tame.  It is not known whether $W$ is all of $\GA_3(k)$.

We will now state and prove our main result:

\begin{theo}[Main Theorem]\label{main0} Let $R$ be a regular ring, $\varphi\in\GA_2(R)$.  Then $\varphi$ is stably tame.
\end{theo}

The Main Theorem is an immediate consequence from the following, thanks to the Jung-Van der Kulk Theorem.

\begin{theo}[First General Form]\label{main} For a fixed integer $n\ge 2$ assume it is true that for all fields $k$ all elements of $\GA_n(k)$ are stably tame.  Then the same is true replacing ``field" by ``regular ring".
\end{theo}

This, in turn, follows from the theorem below,\footnote{The Second General form of the Main Theorem is striking analogous to the main result of \cite{Asa} (Corollary 3.5), which asserts that if $R$ is a regular ring and $A$ is a finitely generated flat $R$-algebra such that $A\otimes k(\mathscr P)\cong k(\mathscr P)^{[n]}$ for all $\mathscr P\in \text{Spec}\,(R)$, then $A$ is stably isomorphic to the symmetric algebra $S_R(P)$ for some projective $R$-aqlgebra $P$.} which employs the following notation: For $R$ a ring and $\mathscr P\in\text{Spec}\,(R)$ we write $k(\mathscr P)$ for the residue field $R_\mathscr P/\mathscr P R_\mathscr P$.  For $\varphi\in\GA_n(R)$, we write $\overline\varphi_\mathscr P$ for the image of $\varphi$ in $\GA_n(k(\mathscr P))$.

\begin{theo}[Second General form]\label{main1}  Let $R$ be a regular ring, $\varphi\in\GA_n(R)$.  Assume $\overline\varphi_\mathscr P$ is stably tame in $\GA_n(k(\mathscr P))$  for all $\mathscr P\in \text{Spec}\,(R)$.  Then $\varphi$ is stably tame.
\end{theo}

This theorem will be proved via a series of reductions, the first being the reduction to the local case.  For this we make the following definition:

\begin{defi} For any ring $R$, an automorphism $\varphi\in\GA_n(R)$ will be called {\it locally tame} if for all prime ideals $\mathfrak p\subset R$, $\varphi_{\mathfrak p}\in\T_n(R_{\mathfrak p})$.  Also, $\varphi$ is called {\it locally stably tame} if for all prime ideals $\mathfrak p$, $\varphi_{\mathfrak p}$ is stably tame.
\end{defi}

The main tool, of considerable interest in itself, will be:

\begin{theo}[Localization Theorem]\label{loc} Let $R$ be a ring, $\varphi\in\GA_n(R)$.  If $\varphi$ is locally tame, then $\varphi$ is stably tame.  
\end{theo}

\begin{rema} \label{locstab}  If $\varphi_{\mathfrak p}\in\T_n(R_{\mathfrak p})$ then it is a routine exercise to see that there exists $a\in R-\mathfrak p$ such that $\varphi_a\in\T_n(R_a)$.  Thus we can find $a_1,\ldots,a_r\in R$ generating $R$ as an ideal such that $\varphi_{a_i}\in\T_n(R_{a_i})$ for each $i$.   It follows that if $\varphi$ is locally stably tame, then $\varphi$ is stably tame.  Just use this observation to bound the number of variables needed at any prime ideal, then stabilize and apply Theorem \ref{loc}.
\end{rema}

We will now prove Theorem \ref{main1} assuming Theorem \ref{loc}.  Appealing to Remark \ref{locstab}, we may assume that $R$ is a regular \underbar{local} ring, since all residue fields of localizations of $R$ are residue fields of $R$.  We proceed by induction on $d=\dim R$.  If $d=0$, $R$ is a field and $\varphi$ is stably tame by hypothesis.

Assume $d\ge 1$.  Since a regular local ring is a domain, we have $|J(\varphi)|\in R^*$, so we may assume $\varphi\in\SA_n(R)$.  Let $a\in R$ be part of a regular system of parameters.  Then $R_a=R[1/a]$ is regular of dimension $d-1$, so all of its localizations are regular local rings of dimension $<d$, so by induction on $d$, appealing to Theorem \ref{loc} and Remark \ref{locstab}, $\varphi_a$ is stably tame.  Note that $\bar R=R/aR$ is a regular local ring of dimension $d-1$, so $\bar\varphi$ is also stably tame by induction ($\bar\varphi$ satisfies the hypothesis of Theorem \ref{main1} since all residue fields of $\bar R$ are residue fields of $R$).  By Proposition \ref{elem}, $\T_m(\bar R)\cap\SA_m(\bar R)=\EA_m(\bar R)$ for all $m$, so $\bar\varphi\in\EA_m(\bar R)$ for some $m\ge n$.  Replacing $n$ by $m$, Theorem \ref{locmod} applies to yield that $\varphi$ is stably tame.  


Thus we are reduced to proving Theorem \ref{loc}.

As we observed in Remark \ref{locstab}, the hypothesis of Theorem \ref{loc} implies the existence of $a_1,\ldots,a_r\in R$ generating $R$ as an ideal such that $\varphi_{a_i}\in\T_n(R_{a_i})$ for $i=1,\ldots,r$.  We will now employ an old technique which reduces to the case $r=2$.  We will show that the set $$J=\{a\in R\,|\,\varphi_a\text{ is stably tame in }\GA_n(R_a)\}$$ is an ideal in $R$.  Since $a_1,\ldots,a_r\in J$, this will show $J=R$, so $1\in J$, so $\varphi_1=\varphi$ is stably tame.  First note that if $x\in R$, $a\in J$, then $xa\in J$, since $R_{xa}$ is a localization of $R_a$.  So it remains to show that $a+b\in J$ when $a,b\in J$.  Note that $\varphi_{a(a+b)}$ and $\varphi_{b(a+b)}$ are both stably tame, being localizations of $\varphi_a$ and $\varphi_b$, respectively, and that $a,b$ generate $R_{a+b}$ as an ideal.  So we are reduced (after stabilizing) to the case $r=2$, i.e., to proving:

\begin{lemma}\label{2cover}  Suppose $R$ a ring and $a,b\in R$ with $aR+bR=R$.  If $\varphi\in\GA_n(R)$ has the property that $\varphi_a\in\T_n(R_a)$ and $\varphi_b\in\T_n(R_b)$, then $\varphi$ is stably tame.
\end{lemma}

If $\varphi=(F_1,\ldots,F_n)\in\GA_n(R)$, let $\rho$ be the translation $X-F(0)$.  Then $\rho\varphi\in\GA_n^0(R)$.  Moreover, $\varphi\in\T_n(R)$ if and only if $\rho\varphi\in\T_n^0(R)$.  So to prove Lemma \ref{2cover} we may assume $\varphi\in\GA_n^0(R)$.   This allows us to use the scalar operator introduced in \ref{scalarop}.  Lemma \ref{2cover} will follow from:

\begin{lemma}\label{cong} Let $R$ be a ring, $t\in R$, $\varphi\in\GA_n^0(R)$.  Assume $\varphi_t\in\T_n^0(R_t)$.  Then there exists an integer $N\ge0$ such that if $c,d\in R$ with $c\equiv d\mod t^NR$, then $\varphi^c(\varphi^d)^{-1}$ is stably tame (over $R$).
\end{lemma}

\begin{rema} This lemma is inspired by an argument introduced by Quillen and Suslin, independently, in their proof of Serre's Conjecture, and in Suslin's follow-up work on $K_1$.  
\end{rema}

We first prove Lemma \ref{2cover} assuming Lemma \ref{cong}.  We may assume $\varphi^0=\text{id}$ (replace $\varphi$ by $\varphi(\varphi^0)^{-1}$, which is valid since $\varphi^0\in\GL_n(R)$).  Let $N$ be the maximum of the integers yielded by Lemma \ref{cong} for $t=a$ and $t=b$, respectively.  Since $a^N$ and $b^N$ generate $R$ as an ideal, then by the Chinese Remainder Theorem we can choose $x\in R$ such that $x\equiv 0\mod a^N$ and $x\equiv 1\mod b^N$.  By Lemma \ref{cong}, $\varphi^1(\varphi^x)^{-1}$ and $\varphi^x(\varphi^0)^{-1}$ are tame.  Their product is $\varphi$, so we are done.

We are now reduced to proving Lemma \ref{cong}.  To this end we introduce indeterminates $W$ and $Z$, which will serve basically as place-markers, and extend $R$ and $R_t$ to $R[W,Z]$ and $R_t[W,Z]$, and consider the automorphism $\psi=\psi(W,Z)=\varphi^{W+Z}(\varphi^W)^{-1}\in\GA_n^0(R[W,Z])$.  

\begin{claim}\label{claim}  For $N$ sufficiently large, $\psi(W,t^NZ)$ is stably tame over $R[W,Z]$.
\end{claim}

This proves Lemma \ref{cong} as follows:  Write $c=d+t^Nb$, and note that $\psi(d,t^Nb)$ is stably tame over $R$, being a specialization of $\psi(W,t^NZ)$, and that it equals $\varphi^c(\varphi^d)^{-1}$.

So now we proceed to prove the claim.  Note that $\psi(W,0)=\text{id}$, i.e., $\psi$ has the form $X+ZH$.  Since $\varphi_t\in\T_n^0(R_t)=\langle\,\GL_n(R_t)\,,\,\EA_n^0(R_t)\,\rangle$, and since the operations $*^{W+Z}$ and $*^W$ clearly carry $\EA_n^0(R_t)$ into $\EA_n^0(R_t[W,Z])$ and fix elements of $\GL_n(R_t)$, we have $\psi_t(W,Z)\in\langle\,\GL_n(R_t)\,,\,\EA_n^0(R_t[W,Z])\,\rangle$.  Additionally, if $\rho$ is elementary and origin preserving over $R_t[W,Z]$, say $\rho=e_i(f)$ with $f\in R_t[W,Z][\Xmi]$, then by writing $f=g+Zh$ with $g\in R_t[W][\Xmi]$, $h\in R_t[W,Z][\Xmi]$, we can write $\rho=\sigma\varepsilon$ with $\sigma=e_i(g)$, $\varepsilon=e_i(Zh)$.  Note that $\sigma$ is elementary over $R_t[W]$ and $\varepsilon$ is elementary over $R_t[W,Z]$, both origin preserving, and that $\varepsilon$ is $Z$-vanishing.  Thus we see that $\psi_t$ lies in the group generated by $\T_n^0(R_t[W])$  together with the origin preserving, $Z$-vanishing elementary automorphisms over $R_t[W,Z]$.  Therefore we can write $\psi_t=\tau_1\varepsilon_1\cdots\tau_r\varepsilon_r$ with $\tau_1,\ldots,\tau_r\in\T_n^0(R_t[W])$ and $\varepsilon_1,\ldots,\varepsilon_r$ elementary, origin preserving, and $Z$-vanishing over $R_t[W,Z]$.  We then have $\tau_1\cdots\tau_r=\psi_t(W,0)=\id$, and we therefore have 
\begin{equation}\label{conjfact}\psi_t=\left(\tau_1\varepsilon_1\tau_1^{-1}\right)\left(\tau_1\tau_2\varepsilon_2(\tau_1\tau_2)^{-1}\right)\cdots\left(\tau_1\tau_2\cdots\tau_r\varepsilon_r(\tau_1\tau_2\cdots\tau_r)^{-1}\right)\end{equation}
The claim will be a consequence of the following lemma (replacing its $R$ by the current $R[W]$).

\begin{lemma}\label{final} Suppose $R$ is a ring, $t\in R$, and $Z$ an indeterminate.  Let $\varepsilon=\varepsilon(Z)$ be an elementary, origin-preserving, and $Z$-vanishing automorphism over $R_t[Z]$, and let $\tau\in\T_n^0(R_t)$.  Then for $N$ sufficiently large, $\tau\varepsilon(t^NZ)\tau^{-1}$ can be written as $\phi_t$ where $\phi$ lies in $\GA_n^0(R[Z])$ and is stably tame over $R[Z]$. More strongly, $\phi$ lies in $\EA_{n+p}^0(R[Z])$ for some $p$.
\end{lemma}

To see that Lemma \ref{final} implies the claim, note that, applying the lemma to the factors on the right side of equation (\ref{conjfact}), we can produce $\phi\in\GA_n^0(R[W,Z])$ $\cap\,\EA_{n+p}^0(R[W,Z])$ such that $\psi_t(W,t^NZ)=\phi_t$.  If $t$ is not a zero divisor, we are done, since in that case the localization homomorphism $R\to R_t$ is injective, so we can conclude $\psi(W,t^NZ)=\phi$.  In the general case, we can replace $\phi=\phi(W,Z)$ by $\phi(W,Z)\phi(W,0)^{-1}$ to arrange that $\phi$ is $Z$-vanishing (and again in $\GA_n^0(R[W,Z])\cap\EA_{n+p}^0(R[W,Z])$), and since $\psi$ is $Z$-vanishing, we still have $\psi_t(W,t^NZ)=\phi_t$.  Now by Lemma \ref{loclem} we have $\psi(W,t^{N+M}Z)=\phi(W,t^MZ)$ for $M$ sufficiently large, which proves the claim.

Now we prove Lemma \ref{final}.  If $\tau\in T_n^0(R_t)$ we can write $\tau=\gamma_1\ldots\gamma_r$, where, for each $i$, $\gamma_i$ is elementary and origin preserving over $R_t$ or $\gamma_i\in\GL_n(R_t)$. We will use induction on $r$.  We now introduce a new variable $T$ and consider $\varepsilon(TZ)$. The case $r=1$ follows directly from Lemma \ref{linelemconj}, applied to $R_t$ instead of $R$, by substituting $t^NZ$ for $Z$ and $t^N$ for $T$ with $N$ sufficiently large. So let $r\geq 2$. Put $\gamma=\gamma_r$ and  $\tau'=\gamma_1\ldots \gamma_{r-1}$. So $\tau=\tau'\gamma$ and $\tau\varepsilon(TZ){\tau}^{-1}=\tau'(\gamma\varepsilon(TZ)\gamma^{-1}){\tau'}^{-1}$. Going up one dimension to $\GA_{n+1}(R_t[Z,T])$ the same equation becomes $$(\tau\varepsilon(TZ){\tau}^{-1})^{[1]}=\tau'^{[1]}(\gamma\varepsilon(TZ)\gamma^{-1})^{[1]}({\tau'^{[1]}})^{-1}\,.$$  By Lemma \ref{linelemconj} $(\gamma\varepsilon(TZ){\gamma}^{-1})^{[1]}=\omega_1\cdots\omega_s$, where each
$\omega_i=\omega_i(T,Z)$ is either a $T$-vanishing or a $Z$-vanishing elementary origin preserving element of $EA_{n+1}(R_t[Z,T])$.
Observe that  $$\tau'^{[1]}(\omega_1\cdots\omega_s){(\tau'^{[1]})}^{-1}=(\tau'^{[1]}\omega_1{(\tau'^{[1]})}^{-1})\cdots 
(\tau'^{[1]}\omega_s{(\tau'^{[1]})}^{-1}).$$  If $\omega_i$ is $T$-vanishing it follows from the induction hypothesis, applied to the ring $R[Z]$ instead of $R$, that there exists $p_i\geq 1$ such that for sufficiently large $N$ 
$$
\left(\tau'^{[1]}\omega_i(t^NT,Z)(\tau'^{[1]})^{-1}\right)^{[p_i]}\text{ lifts to }EA_{n+1+p_i}(R[Z][T]).
$$
Similarly, if $\omega_i$ is $Z$-vanishing we apply the induction hypothesis to the ring $R[T]$ to see that there exists $p_i\geq 1$ such that for sufficiently large $N$
$$
\left(\tau'^{[1]}\omega_i(T,t^NZ)(\tau'^{[1]})^{-1}\right)^{[p_i]}\text{ lifts to }EA_{n+1+p_i}(R[Z][T]).
$$  
Taking $p$ to be the maximum of all $p_i$, then for sufficiently large $N$ each of the automorphisms $(\tau'^{[1]}\omega_i(t^NZ, t^NT)(\tau'^{[1]})^{-1})^{[p]}$ lifts to $EA_{n+1+p}(R[Z,T])$.  Setting $T=1$ we obtain that $\tau\varepsilon(t^{2N}Z)\tau^{-1})^{[p+1]}$ lifts to $EA_{n+1+p}(R[Z])$,
as desired. 

This completes the proof of Theorem \ref{main1}, and hence of Theorems \ref{main} and \ref{main0}.

\section{Further Remarks and Conclusions}\label{further}  In \cite{S-U} it was shown that there exist many non-tame automorphisms which fix one variable in dimension three over a field of characteristic zero.  Nevertheless, Corollary \ref{cor} shows that all such automorphisms are stably tame.  Consequently, if $\GA_3(k)$, for $k$ a field, were generated by $\GL_3(k)$ together with the automorphisms that fix one variable, then all elements of $\GA_3(k)$ would be stably tame.  This raises another question, for which we make the following definition:

\begin{defi}\label{weak}  For $R$ a ring we say that an automorphism $\phi\in\GA_n(R)$ is {\it weakly tame} if it is in the subgroup generated by $\Af_n(R)$ together with all automorphisms which fix one variable.  We will denote the subgroup of weakly tame automorphisms by $\WT_n(R)$.
\end{defi}

Now we can pose:

\begin{quest}[Weak Generators Problem]\label{quest}  {\bf WGP(n):}  Let $k$ be a field and  $n$ an integer $\ge1$.  Are all $n$-dimensional automorphisms weakly tame, i.e., is $\WT_n(k)=\GA_n(k)$?
\end{quest}

For $n=1$ this is trivially affirmative (note that the use of $\Af$ rather than $\GL$ in \ref{weak} assures this).  If $n=2$ a positive answer follows easily from the Jung-Van der Kulk Theorem (Theorem \ref{JvdK}).  However, for $n\ge3$ the problem remains open.

Note that ``stable weak tameness" holds no interest, since all automorphisms become weakly tame upon adding one new dimension.  However the following theorem relates weak tameness to stable tameness:

\begin{theo}  Let $n\ge1$ be fixed.  If $\text{WGP}(m)$ has an affirmative answer for $1\le m\le n$, then, for any regular ring $R$, all elements of $\GA_n(R)$ are stably tame.
\end{theo}

\begin{proof}We proceed by induction on $n$.  The case $n\le2$ is known by Theorem \ref{main0}, so let $n\ge3$ and assume the theorem holds for integers $<n$. Assume $\text{WGP}(m)$ holds for $1\le m\le n$.  The induction hypothesis tells us that elements of $\GA_m(R)$ are stably tame, for $R$ a regular ring and $1\le m <n$.  

By Theorem \ref{main} it suffices to show all elements of $\GA_n(k)$ are stably tame for $k$ a field.  So let $\varphi\in\GA_n(k)$, and by the hypothesis we may assume that $\varphi$ fixes one variable, say $X_1$.  Letting $R=k[X_1]$, we therefore have $\varphi\in\GA_{n-1}(R)$.  Since $R$ is regular, the last assertion in the previous paragraph yields that $\varphi$ is stably tame over $R$, hence over $k$ (see Remark \ref{eucl}).
\end{proof}

\bibliographystyle{amsplain}
\bibliography{Refs}

\vskip 15pt
\noindent{\small \sc Department of Mathematics, Washington University in St.
Louis,
St. Louis, MO 63130 } 
{\em E-mail}: wright@math.wustl.edu

\noindent{\small \sc Department of Mathematics, Radboud University,
Nijmegen, The Netherlands } {\em E-mail}: j.berson@science.ru.nl, essen@math.ru.nl

\end{document}